\documentclass{article}

\usepackage{amsmath}
\usepackage{amssymb}
\usepackage{amsthm}
\usepackage{authblk}
\usepackage[shortlabels]{enumitem}
\usepackage{tikz-cd}

\newtheorem{thm}{Theorem}[section]

\newtheorem{lem}[thm]{Lemma}

\theoremstyle{definition}

\theoremstyle{remark}
\newtheorem{rem}{Remark}

\begin{document}

\title{Degrees of maps between $S^3$-bundles over $S^5$}
\author{Xueqi Wang}
\affil{Institute of Mathematics, Chinese Academy of Sciences, Bejing, 100190, China \authorcr Email: wangxueqi@amss.ac.cn}
\date{}
\maketitle

\begin{abstract}
In this article, we compute all possible degrees of maps between $S^3$-bundles over $S^5$. It also provides a correction of an article by Lafont and Neofytidis \cite{LN17}.


\end{abstract}

\section{The main result}

For two $n$-dimensional closed manifold $M$ and $N$, let $D(M,N)$ be the set of degrees of maps from $M$ to $N$, i.e.
\[
D(M,N)=\{\deg f| f:M\to N\}.
\]
Set $D(M)=D(M,M)$ for short.

The set of mapping degrees has been studied by many authors, for both lower dimensional and higher dimensional manifolds, cf. \cite{DW03,GVC17,SWW10} etc. Unlike the complicated situation for lower dimensional manifolds, homotopy theory seems to be a powerful tool in the higher dimensional case.

In this article, the main objects concerned are $S^3$-bundles over $S^5$. Recall that isomorphism classes of $4$-dimensional vector bundles over $S^5$ are one-to-one correspondence with elements in $\pi_4(SO(4))$. Using the canonical diffeomorphism $SO(4)\cong S^3\times SO(3)$, we have $\pi_4(SO(4))\cong \pi_4(S^3)\oplus\pi_4(SO(3))\cong \mathbb{Z}_2\oplus\mathbb{Z}_2$. Under this correspondence, denote the total spaces of those associated sphere bundles as $M_{i,j}$ ($i=0,1$). Obviously, $M_{0,0}\cong S^3\times S^5$, $M_{1,0}\cong SU(3)$. Moreover, we have:

\begin{lem}[\cite{JW54,JW55}]
  \begin{enumerate}
    \item $S^3\times S^5$, $M_{0,1}$ and $SU(3)$ are not homotopy equivalent to each other.
    \item $M_{1,1}\simeq SU(3)$.
  \end{enumerate}
\end{lem}

Therefore, we only need to consider three objects: $S^3\times S^5$, $M_{0,1}$ and $SU(3)$. The main result is the following:

\begin{thm}\label{thm:2}
  Let $M,N\in \{S^3\times S^5, M_{0,1}, SU(3)\}$ and $M\neq N$, then
  \begin{enumerate}
    \item $D(M)=
    \begin{cases}
      \mathbb{Z}, & \mbox{if } M=S^3\times S^5, M_{0,1} \\
      4\mathbb{Z}\cup\{2k+1|k\in\mathbb{Z}\}, & \mbox{if } M=SU(3).
    \end{cases}$
    \item $D(M,N)=
    \begin{cases}
      4\mathbb{Z}, & \mbox{if } M=SU(3) \\
      2\mathbb{Z}, & \mbox{otherwise}.
    \end{cases}$
  \end{enumerate}
\end{thm}

\begin{rem}
  The motivation comes from the work of Lafont and Neofytidis \cite{LN17}. They considered mapping degrees among principal $SU(2)$-bundles over $S^5$. Note that $SU(2)\cong S^3$. One aim of their article is to correct P\"utmann's result of $D(SU(3))$ \cite{Put09}. However, their result about $D(SU(3), S^3\times S^5)$ is wrong. The reason is that the image of their map $g$ does not lie in $SU(3)$. For example, choose $A=
  \begin{pmatrix}
    0 & 0 & 1 \\
    1 & 0 & 0 \\
    0 & 1 & 0
  \end{pmatrix}\in SU(3)$, then $g(A)=
  \begin{pmatrix}
    0 & 0 & 0 \\
    0 & 0 & -1 \\
    0 & 0 & 0
  \end{pmatrix}\notin SU(3)$. It also leads to a wrong proof of their result on $D(SU(3))$, which relies on $D(SU(3), S^3\times S^5)$ in their article, although the final result is miraculously correct.
\end{rem}

\begin{rem}
  In fact, any simply connected $8$-manifold with homology isomorphic to that of $S^3\times S^5$ is homotopy equivalent to either $S^3\times S^5$, $M_{0,1}$ or $SU(3)$. It will be proved in a coming paper of the author. Therefore we actually obtained the set of mapping degrees among a larger class of manifolds.
\end{rem}

\section{The proof}

We'll analysis the homotopy sets between these manifolds. First recall some basic results for homotopy groups of spheres. Let $\iota_n\in\pi_n(S^n)$ be represented by the identity map of $S^n$, $\eta_2\in\pi_3(S^2)$ and $\nu_4\in\pi_7(S^4)$ be represented by Hopf maps, $\eta_n=\Sigma^{n-2}\eta_2\in\pi_{n+1}(S^n)$.

\begin{lem}[\cite{Tod52,Tod62}]
  \begin{enumerate}
    \item $\pi_{n+1}(S^n)\cong
    \begin{cases}
      \mathbb{Z}\{\eta_2\}, & \mbox{if } n=2 \\
      \mathbb{Z}_2\{\eta_n\}, & \mbox{if } n\geq 3,
    \end{cases}$
    \item $\pi_{n+2}(S^n)\cong\mathbb{Z}_2\{\eta_n\eta_{n+1}\}$ for $n\geq 2$;
    \item $\pi_6(S^3)\cong \mathbb{Z}_{12}\{a_3\}$, \\
    $\pi_7(S^4)\cong \mathbb{Z}\{\nu_4\}\oplus\mathbb{Z}_{12}\{a_4\}$, $a_4=\Sigma a_3$, $[\iota_4,\iota_4]=2\nu_4\pm a_4$;
    \item $\pi_7(S^3)\cong\mathbb{Z}_2\{a_3\eta_6\}$,\\
    $\pi_8(S^4)\cong \mathbb{Z}_2\{\nu_4\eta_7\}\oplus\mathbb{Z}_2\{a_4\eta_7\}$.
  \end{enumerate}
\end{lem}

Next write down the cell structures.

\begin{lem}[cf. \cite{JW54,Muk82}]
  \begin{enumerate}
    \item $S^3\times S^5\simeq (S^3\vee S^5)\cup_{[\iota_3,\iota_5]} D^8$;
    \item $M_{0,1}\simeq (S^3\vee S^5)\cup_{[\iota_3,\iota_5]+a_3\eta_6} D^8$;
    \item $SU(3)\simeq S^3\cup_{\eta_3}D^5\cup_{\xi}D^8$, where $\xi$ generates $\pi_7(S^3\cup_{\eta_3}D^5)\cong\mathbb{Z}$.
  \end{enumerate}
\end{lem}

Note that in this article, obvious inclusion maps are always omitted.

In the following, for $M\in\{S^3\times S^5, M_{0,1}, SU(3)\}$, choose $e_3^M$ and $e_5^M$ to be generators of $H^3(M)\cong\mathbb{Z}$ and $H^5(M)\cong\mathbb{Z}$, such that $e_3^Me_5^M$ equals to the orientation cohomology class $\omega_M$ of $M$.

\begin{proof}[Proof of theorem \ref{thm:2}]
  \begin{enumerate}[(1),wide]
  \item $D(S^3\times S^5)=\mathbb{Z}$ is obvious.
  \item $D(S^3\times S^5,M_{0,1})=2\mathbb{Z}$.

  Consider the Puppe sequence induced by the cofiber sequence $S^7\xrightarrow{[\iota_3,\iota_5]} S^3\vee S^5\to S^3\times S^5$:
  \[
  [S^3\times S^5,M_{0,1}]\to [S^3\vee S^5,M_{0,1}]\xrightarrow{[\iota_3,\iota_5]^*} [S^7,M_{0,1}]
  \]
  $[S^3\vee S^5,M_{0,1}]\cong [S^3,M_{0,1}]\times [S^5,M_{0,1}]\cong \mathbb{Z}\{\iota_3\}\oplus\mathbb{Z}_2\{\eta_3\eta_4\}\oplus\mathbb{Z}\{\iota_5\}$. An element can be lifted to $[S^3\times S^5,M_{0,1}]$ if and only if it maps to $0$ under $[\iota_3,\iota_5]^*$. We have
  \[
  \begin{split}
    [\iota_3,\iota_5]^*(k\iota_3+\epsilon\eta_3\eta_4+l\iota_5) & =(k\iota_3+\epsilon\eta_3\eta_4+l\iota_5)[\iota_3,\iota_5] \\
      & =[k\iota_3,\epsilon\eta_3\eta_4+l\iota_5] \\
      & =kl[\iota_3,\iota_5]+k\epsilon[\iota_3,\eta_3\eta_4] \\
      & =kla_3\eta_6
  \end{split}
  \]
  In the last step, $[\iota_3,\iota_5]=a_3\eta_6$ in $M_{0,1}$ and $[\iota_3,\eta_3\eta_4]=0$ as $S^3$ is a Lie group \cite[Corollary X.7.8]{Whi78}. Since $a_3\eta_6$ is of order $2$, $[\iota_3,\iota_5]^*(k\iota_3+\epsilon\eta_3\eta_4+l\iota_5)=0$ if and only if $kl$ is even. Notice that $kl$ is just the degree of the lifted map. Therefore the set of degrees coincides with all even numbers.
  \item\label{item:1} $D(S^3\times S^5,SU(3))=2\mathbb{Z}$.

  Consider the Puppe sequence induced by the cofiber sequence $S^7\xrightarrow{[\iota_3,\iota_5]} S^3\vee S^5\to S^3\times S^5$:
  \[
  [S^3\times S^5,SU(3)]\to [S^3\vee S^5,SU(3)]\xrightarrow{[\iota_3,\iota_5]^*} [S^7,SU(3)]
  \]
  By the cell structure of $SU(3)$, $\pi_7(SU(3))=0$. Therefore, all elements in $[S^3\vee S^5,SU(3)]$ can be lifted to $[S^3\times S^5,SU(3)]$. $[S^3\vee S^5,SU(3)]\cong [S^3,SU(3)]\times [S^5,SU(3)]\cong \mathbb{Z}\{\alpha\}\oplus\mathbb{Z}\{\beta\}$. Suppose $(k\alpha,l\beta)$ lifts to a map $f_{k,l}:S^3\times S^5\to SU(3)$. Then we have
  \[
  f_{k,l}^*\omega_{SU(3)}=f_{k,l}^*(e_3^{SU(3)}e_5^{SU(3)})= f_{k,l}^*e_3^{SU(3)}f_{k,l}^*e_5^{SU(3)}
  \]
  It is clear that the value of $f_{k,l}^*e_3^{SU(3)}$ and $f_{k,l}^*e_5^{SU(3)}$ correspond to the image of $k\alpha$ and $l\beta$ under the Hurwicz maps.
  By Hurewicz theorem, $h_3: \pi_3(SU(3))\to H_3(SU(3))$ is an isomorphism. To compute $h_5:\pi_5(SU(3))\to H_5(SU(3))$, use the fiber bundle $S^3\to SU(3)\xrightarrow{p} S^5$:
  \[
  \begin{tikzcd}
  \pi_5(SU(3)) \arrow[r,"p_*"] \arrow[d,"h_5"] & \pi_5(S^5) \arrow[r] \arrow[d,"h_5","\cong"'] & \pi_4(S^3) \arrow[r] & \pi_4(SU(3))=0 \\
  H_5(SU(3)) \arrow[r,"p_*","\cong"'] & H_5(S^5) & &
  \end{tikzcd}
  \]
  The bottom isomorphism follows easily by the Gysin sequence. Therefore $h_5:\pi_5(SU(3))\to H_5(SU(3))$ is a multiplication by $\pm2$ as $\pi_4(S^3)\cong \mathbb{Z}_2$. With those generators suitably chosen, we may assume the sign is $+$. Hence
  \[
  f_{k,l}^*\omega_{SU(3)} =f_{k,l}^*e_3^{SU(3)}f_{k,l}^*e_5^{SU(3)} =2kle_3^{S^3\times S^5}e_5^{S^3\times S^5}=2kl\omega_{S^3\times S^5}
  \]
  which means that the set of degrees is just all even numbers.
  \item $D(M_{0,1},S^3\times S^5)=2\mathbb{Z}$.

  Consider the cofiber sequence $S^7\xrightarrow{[\iota_3,\iota_5]+a_3\eta_6} S^3\vee S^5\to M_{0,1}$. We have an exact sequence
  \[
  [M_{0,1},S^3\times S^5]\to [S^3\vee S^5,S^3\times S^5]\xrightarrow{([\iota_3,\iota_5]+a_3\eta_6)^*} [S^7,S^3\times S^5]
  \]
  $[S^3\vee S^5,S^3\times S^5]\cong [S^3,S^3\times S^5]\times [S^5,S^3\times S^5]\cong \mathbb{Z}\{\iota_3\}\oplus\mathbb{Z}_2\{\eta_3\eta_4\}\oplus\mathbb{Z}\{\iota_5\}$.
  \[
  \begin{split}
    ([\iota_3,\iota_5]+a_3\eta_6)^*(k\iota_3+\epsilon\eta_3\eta_4+l\iota_5) & =(k\iota_3+\epsilon\eta_3\eta_4+l\iota_5)([\iota_3,\iota_5]+a_3\eta_6) \\
      & =[k\iota_3,\epsilon\eta_3\eta_4+l\iota_5]+ka_3\eta_6 \\
      & =kl[\iota_3,\iota_5]+k\epsilon[\iota_3,\eta_3\eta_4]+ka_3\eta_6 \\
      & =ka_3\eta_6
  \end{split}
  \]
  In the last step, note that $[\iota_3,\iota_5]=0$ in $S^3\times S^5$. Since $ka_3\eta_6=0$ if and only if $k$ is even, the possible degrees $kl$ are chosen from all even numbers.
  \item $D(M_{0,1},M_{0,1})=\mathbb{Z}$.

  Consider the cofiber sequence $S^7\xrightarrow{[\iota_3,\iota_5]+a_3\eta_6} S^3\vee S^5\to M_{0,1}$. The we have an exact sequence:
  \[
  [M_{0,1},M_{0,1}]\to [S^3\vee S^5,M_{0,1}] \xrightarrow{([\iota_3,\iota_5]+a_3\eta_6)^*} [S^7,M_{0,1}]
  \]
  \[
  \begin{split}
    ([\iota_3,\iota_5]+a_3\eta_6)^*(k\iota_3+\epsilon\eta_3\eta_4+l\iota_5) & =(k\iota_3+\epsilon\eta_3\eta_4+l\iota_5)([\iota_3,\iota_5]+a_3\eta_6) \\
      & =[k\iota_3,\epsilon\eta_3\eta_4+l\iota_5]+ka_3\eta_6 \\
      & =k(l+1)a_3\eta_6
  \end{split}
  \]
  which equals to $0$ if and only if $k(l+1)$ is even. Thus we can choose $l=1$ and $k$ arbitrary to make the mapping degree $kl$ realize all integers.
  \item $D(M_{0,1},SU(3))=2\mathbb{Z}$.

  The proof is the same as \ref{item:1} except that $S^3\times S^5$ is changed to $M_{0,1}$.
  \item\label{item:2} $D(SU(3),S^3\times S^5)=4\mathbb{Z}$.

  We have $[SU(3),S^3\times S^5]=[SU(3),S^3]\times [SU(3),S^5]$. For $[SU(3),S^3]$, first consider the Puppe sequence induced by the cofiber sequence $S^4\xrightarrow{\eta_3} S^3 \xrightarrow{i} S^3\cup_{\eta_3} D^5$:
  \[
  [S^4,S^3]\xrightarrow{\eta_4^*} [S^5,S^3] \to [S^3\cup_{\eta_3} D^5, S^3] \xrightarrow{i^*} [S^3,S^3] \xrightarrow{\eta_3^*} [S^4,S^3]
  \]
  It is clear that $\eta_4^*$ is an isomorphism and $\eta_3^*$ is surjective. Therefore $[S^3\cup_{\eta_3}D^5, S^3] \xrightarrow[\cong]{i^*} 2[S^3,S^3]$. Suppose $f_k\in[S^3\cup_{\eta_3} D^5, S^3]$ maps to $2k\iota_3$ under $i^*$. Then consider the Puppe sequence induced by the cofiber sequence $S^7\xrightarrow{\xi} S^3\cup_{\eta_3} D^5 \to SU(3)$:
  \[
  [SU(3),S^3] \to [S^3\cup_{\eta_3} D^5, S^3] \xrightarrow{\xi^*} [S^7,S^3]
  \]
  We want to know that for which $k$, $\xi^*f_k=f_k\xi=0$. Notice that the suspension map $\Sigma:\pi_7(S^3)\to \pi_8(S^4)$ is injective. Hence $f_k\xi=0$ is equivalent to $\Sigma(f_k\xi)=0$.
  \[
  \begin{split}
    \Sigma(f_k\xi) & =\Sigma f_k\Sigma\xi \\
      & =(\Sigma f_k)\nu_4\eta_7 \quad \text{($\Sigma\xi=\nu_4\eta_7$ \cite{Muk82})}\\
      & =\Sigma (f_ki)\nu_4\eta_7 \\
      & =(2k\iota_4)\nu_4\eta_7 \\
      & =(2k\nu_4+k(2k-1)[\iota_4,\iota_4])\eta_7 \quad \text{by Hilton's formula \cite{Hil55}} \\
      & =(2k\nu_4+k(2k-1)(2\nu_4\pm a_4))\eta_7 \\
      & =ka_4\eta_7
  \end{split}
  \]
  Therefore, $\Sigma(f_k\xi)=0$ if and only if $k$ is even, which implies that the possible mapping degrees are all divisible by $4$. Let $p$ be the bundle projection $SU(3)\to S^5$, then $p_*:H_5(SU(3))\to H_5(S^5)$ is an isomorphism. Therefore, the degree of $f_k\times p:SU(3)\to S^3\times S^5$ is $4k$, and our result follows.
  \item $D(SU(3),M_{0,1})=4\mathbb{Z}$.

  It's not as straight as the other cases. We'll use the following observation.

  \begin{lem}[\cite{Tan06}]\label{thm:1}
    Let $M,N$ be two closed $n$-manifolds and $\overline{M}, \overline{N}$ be obtained by deleting an embedded $D^n$ in $M,N$ respectively. Then $k\in D(M,N)$ if and only if there exists $\overline{f}:\overline{M}\to \overline{N}$ such that the following diagram commutes up to homotopy:
    \[
    \begin{tikzcd}
    S^{n-1} \arrow[r,"i_M"] \arrow[d,"k\iota_{n-1}"] & \overline{M} \arrow[d,"\overline{f}"] \\
    S^{n-1} \arrow[r,"i_N"] & \overline{N}
    \end{tikzcd}
    \]
    Here $i_M, i_N$ denotes the inclusion of $S^{n-1}$ to the boundary of $\overline{M},\overline{N}$ respectively.
  \end{lem}

  Consider the Puppe sequence induced by the cofiber sequence $S^4\xrightarrow{\eta_3} S^3 \xrightarrow{i} S^3\cup_{\eta_3} D^5$:
  \[
  [S^3\cup_{\eta_3} D^5, M_{0,1}] \xrightarrow{i^*} [S^3,M_{0,1}] \xrightarrow{\eta_3^*} [S^4,M_{0,1}]
  \]
  It is clear that $\eta_3:\pi_3(M_{0,1})\cong\mathbb{Z}\{\iota_3\}\to \pi_4(M_{0,1})\cong \mathbb{Z}_2\{\eta_3\}$ is surjective. Therefore, for any map $f:S^3\cup_{\eta_3} D^5 \to M_{0,1}$, if $f_*:H_3(S^3\cup_{\eta_3} D^5)\cong \mathbb{Z} \to H_3(M_{0,1})\cong \mathbb{Z}$ is a multiplication by $k$, then $k$ must be divisible by $2$. It indicates that all possible mapping degrees from $SU(3)$ to $M_{0,1}$ must be divisible by $2$.

  Now observe that the following two diagrams are equivalent, as $([\iota_3,\iota_5]+a_3\eta_6)(2k\iota_7)=[\iota_3,\iota_5](2k\iota_7)$:
  \[
  \begin{tikzcd}
  S^7 \arrow[r,"\xi"] \arrow[d,"2k\iota_7"] & S^3\cup_{\eta_3} D^5 \arrow[d,"\overline{f}"] \\
  S^7 \arrow[r,"{[\iota_3,\iota_5]}"] & S^3\vee S^5
  \end{tikzcd}
  \qquad
  \begin{tikzcd}[column sep=large]
  S^7 \arrow[r,"\xi"] \arrow[d,"2k\iota_7"] & S^3\cup_{\eta_3} D^5 \arrow[d,"\overline{f}"] \\
  S^7 \arrow[r,"{[\iota_3,\iota_5]}+a_3\eta_6"] & S^3\vee S^5
  \end{tikzcd}
  \]
  Therefore, combining lemma \ref{thm:1} and \ref{item:2}, the result follows.
  \item $D(SU(3))=4\mathbb{Z}\cup\{2k+1|k\in\mathbb{Z}\}$.

  As seen in \ref{item:1}, all maps $S^3\cup_{\eta_3} D^5\to SU(3)$ extends to $SU(3)\to SU(3)$ since $\pi_7(SU(3))=0$. Consider the Puppe sequence induced by the cofiber sequence $S^4\xrightarrow{\eta_3} S^3 \xrightarrow{i} S^3\cup_{\eta_3} D^5$:
  \begin{multline*}
   0=[S^4,SU(3)]\to [S^5,SU(3)] \xrightarrow{q^*} \\
   [S^3\cup_{\eta_3} D^5, SU(3)] \xrightarrow{i^*} [S^3,SU(3)] \to [S^4,SU(3)]=0
  \end{multline*}
  Here $q:S^3\cup_{\eta_3}D^5\to S^5$ is the quotient map. Since $\pi_3(SU(3))\cong \pi_5(SU(3))\cong\mathbb{Z}$ and $[S^3\cup_{\eta_3} D^5, SU(3)]$ is abelian, we have $[S^3\cup_{\eta_3} D^5, SU(3)] \cong\mathbb{Z}\{\alpha\}\oplus\mathbb{Z}\{\beta\}$, where $\alpha$ is the inclusion of $5$-skeleton of $SU(3)$ and $\beta$ is the generator of $\pi_5(SU(3))$ composed with $q$. Obviously,
  \begin{gather*}
    \alpha^*:H^3(SU(3))\xrightarrow{\cong} H^3(S^3\cup_{\eta_3}D^5), \quad H^5(SU(3))\xrightarrow{\cong} H^5(S^3\cup_{\eta_3}D^5) \\
    \beta^*:H^3(SU(3))\xrightarrow{0} H^3(S^3\cup_{\eta_3}D^5), \quad H^5(SU(3))\xrightarrow{\times 2} H^5(S^3\cup_{\eta_3}D^5)
  \end{gather*}
  Therefore,
  \begin{gather*}
    (k\alpha+l\beta)^*:H^3(SU(3))\xrightarrow{\times k} H^3(S^3\cup_{\eta_3}D^5) \\
    (k\alpha+l\beta)^*: H^5(SU(3))\xrightarrow{\times (k+2l)} H^5(S^3\cup_{\eta_3}D^5)
  \end{gather*}
  Extend $k\alpha+l\beta$ to $f_{k,l}:SU(3)\to SU(3)$, then $\deg f_{k,l}=k(k+2l)$. We can realize all odd numbers by choosing $k=1$ and $l$ arbitrary. Notice that if $\deg f_{k,l}$ is even, $k$ must be even. Let $k=2k'$, then $\deg f_{k,l}=4k'(k'+l)$, which is divisible by $4$. Hence we can choose $k'$ arbitrary and $l=1-k'$ to realize $4\mathbb{Z}$.
\end{enumerate}
\end{proof}

\subsection*{Acknowledgment}
The author would like to thank Prof. Haibao Duan for recommending this topic.

\bibliography{bib.bib}

\end{document}